\documentclass[10pt]{article}
\usepackage{graphics}
\usepackage{graphicx}
\usepackage{amsmath}
\usepackage{amsfonts}
\usepackage{amssymb}
\usepackage{amsthm}
\usepackage{latexsym}
\usepackage{hyperref}
   \usepackage{epstopdf}
\usepackage{amsmath}
\usepackage{amssymb}
\usepackage{amsfonts}
\usepackage{amsbsy}
\usepackage{bbm}
\usepackage[]{algorithm2e}

\usepackage{fullpage}

\newtheorem{prop}{Proposition}

\begin{document}

\title{\Large It's good to be $\phi$: a solution to a problem of Gosper and Knuth\footnote{Franklin H. J. Kenter; kenter@usna.edu; 572C Holloway Road, Annapolis, MD 21401; \today}}
\date{}
\maketitle
We present the solution to a problem presented by Knuth, attributed to Gosper \cite{q3}. 

\begin{prop}
Let $\phi$ denote the golden ratio. Then,
\[ \phi = \frac{2^{2/5} \sqrt{5}~\Gamma(\frac1 5)^4}{\Gamma(\frac{1}{10})^2 \Gamma(\frac{3}{10})^2}. \]
\end{prop}

\begin{proof}
The proof relies on two equalities below. These equalities are listed on the {\it MathWorld} entry for the ``Gamma Function'' \cite{q1}; they are attributed to techniques within \cite{q2} and generated by M. Trott \cite{q1}.

\begin{eqnarray}
\Gamma \left(\frac{1}{10}\right) &=& \frac{\sqrt[4]{5}~ \sqrt{1+\sqrt{5}~} ~\Gamma\left(\frac{1}{5}\right) \Gamma \left(\frac{2}{5}\right)}{2^{7/10} \sqrt{\pi }} \label{G1} \\
 \Gamma\left(\frac{3}{10}\right) &=& \frac{\left(\sqrt{5}-1\right) \sqrt{\pi } \Gamma \left(\frac{1}{5}\right)}{2^{3/5} \Gamma \left(\frac{2}{5}\right)} \label{G3} 
\end{eqnarray}
By applying lines \ref{G1} and \ref{G3}, we have:
\begin{eqnarray}
\frac{2^{2/5} \sqrt{5}~ \Gamma \left(\frac{1}{5}\right)^4}{\Gamma(\frac{1}{10})^2 \Gamma(\frac{3}{10})^2} 
&=& \displaystyle \frac{2^{2/5} \sqrt{5}~ \Gamma \left(\frac{1}{5}\right)^4}{\left[\frac{\sqrt[4]{5} \sqrt{1+\sqrt{5}~} \Gamma \left(\frac{1}{5}\right) \Gamma \left(\frac{2}{5}\right)}{2^{7/10} \sqrt{\pi }}\right]^2 \left[\frac{\left(\sqrt{5}-1\right) \sqrt{\pi } \Gamma \left(\frac{1}{5}\right)}{2^{3/5} \Gamma  \left(\frac{2}{5}\right)}\right]^2} \\
&=& \displaystyle \frac{2^{2/5} \sqrt{5}~ \Gamma \left(\frac{1}{5}\right)^4}{\frac{\sqrt{5}~ (1+\sqrt{5}~) \Gamma \left(\frac{1}{5}\right)^2 \Gamma \left(\frac{2}{5}\right)^2 \cdot \left(\sqrt{5}-1\right)^2 \pi \Gamma \left(\frac{1}{5}\right)^2}{2^{7/5} \pi \cdot 2^{6/5} \Gamma  \left(\frac{2}{5}\right)^2 }} \\
&=& \displaystyle \frac{8 }{ (1+\sqrt{5})   \cdot \left(\sqrt{5}-1\right)^2} \\
&=& \displaystyle \frac{2 }{ \sqrt{5}-1 } ~~=~~ \displaystyle \phi 
\end{eqnarray}
\end{proof}

%

\end{document}